\documentclass{article}
\usepackage[utf8]{inputenc}
\usepackage{amsmath,amssymb,amsthm,amsfonts,latexsym,graphicx,subfigure,mathtools,enumitem}

\usepackage{url}
\usepackage{xcolor}
\usepackage{svg}

\newcommand{\R}{\mathbb{R}}

\DeclareMathOperator{\supp}{supp}

\DeclareMathOperator{\Spec}{\text{Spec}}

\theoremstyle{plain}
\newtheorem{theorem}{Theorem}[section]
\newtheorem{proposition}[theorem]{Proposition}
\newtheorem{lemma}[theorem]{Lemma}
\newtheorem{corollary}[theorem]{Corollary}
\newtheorem*{corollary*}{Corollary}

\theoremstyle{definition}
\newtheorem{definition}[theorem]{Definition}
\newtheorem{example}[theorem]{Example}
\theoremstyle{remark}
\newtheorem{remark}[theorem]{Remark}

\usepackage{stmaryrd} 
\SetSymbolFont{stmry}{bold}{U}{stmry}{m}{n}
\usepackage[affil-it]{authblk}
\usepackage{booktabs}   
\usepackage[numbers,sort&compress]{natbib}

\usepackage{comment} 

\allowdisplaybreaks

\usepackage{hyperref}

\usepackage[capitalise,noabbrev]{cleveref}

\usepackage[margin=3.5cm]{geometry}

\title{Spectral Perspectives on FAT Graph Colorings}

\author{Lies Beers\thanks{e.g.m.beers@vu.nl} }
\author{Raffaella Mulas}
\date{}
\affil{Vrije Universiteit Amsterdam}
\begin{document}

\maketitle

\begin{abstract} 
We investigate Fair and Tolerant (FAT) graph colorings, a coloring framework in which each vertex is allowed to share its color with a prescribed fraction of its neighbors, while the remaining neighbors are required to be distributed evenly among the other coloring classes. In particular, we determine the FAT chromatic number for all complete multipartite graphs, and we analyze the behavior of FAT colorings under several graph operations. Although spectral methods form the primary focus, several combinatorial arguments are included to complement the results.
\end{abstract}

\section{Introduction}\label{sec:introduction}

In a recent work \cite{beersmulasFAT}, we introduced the notion of \textit{Fair and Tolerant (FAT) colorings}, in which each vertex is allowed to share its color with a given fraction of its neighbors (tolerance), while the remaining neighbors are required to be distributed evenly among the other coloring classes (fairness). Thus, in contrast to classical graph coloring, FAT colorings do not forbid adjacent vertices from sharing the same color, and additionally prescribe fixed proportions of neighbors in each coloring class. This notion naturally models situations in which strict separation between categories is unnecessary, while a balanced distribution of interactions is still required.\newline 

FAT colorings generalize the well-known \emph{equitable graph colorings} \cite{beersmulas2024}, and at the same time form a special case of both \emph{domatic colorings} \cite{zelinka1981domatic} and \emph{majority colorings} \cite{bosek2019majority}.\newline 

In \cite{beersmulasFAT}, we also introduced the \emph{FAT chromatic number} of a graph $G$, defined as the largest integer $k$ such that $G$ admits a FAT $k$-coloring. In that work, we derived general bounds for the FAT chromatic number, investigated its connections to structural and spectral features of graphs, and determined its value for certain families of graphs. These results, in particular, revealed a strong relationship between FAT colorings and the spectrum of the normalized Laplacian, which motivates the spectral focus in the present paper.\newline 

Moreover, following the introduction of FAT colorings, the notion has already attracted further attention in the literature. In particular, recent preprints by Shaebani \cite{Shaebani1,Shaebani2} investigate relations between the classical chromatic number and the FAT chromatic number, and provide explicit constructions of graphs admitting FAT colorings with fixed parameters. \newline 

In this paper, we further develop the theory of such graph colorings. In particular, using spectral methods, we determine the FAT chromatic number for all complete multipartite graphs, extending \cite{beersmulasFAT}, where this number was determined for the special case of regular Turán graphs using a more involved combinatorial argument. We also analyze the behavior of FAT colorings under several graph operations. While spectral techniques form the primary focus of this work, several combinatorial arguments are included to complement and extend the theory.

\subsection*{Structure of the paper}
In Section \ref{sec:background}, we present the relevant definitions and the background required for the remainder of the paper. In Section \ref{section:Turan}, we revisit the FAT chromatic number of regular Turán graphs and give an alternative spectral proof of the characterization obtained in \cite{beersmulasFAT}, providing a simpler approach than the earlier combinatorial argument. Moreover,
in Section \ref{sec:multipartite}, we generalize the methodology developed in Section \ref{section:Turan} to arbitrary complete multipartite graphs. Finally, in Sections \ref{sec:removing}--\ref{sec:products}, we analyze how various graph operations can be utilized to construct new FAT colorings. We begin by removing FAT coloring classes of a graph, corresponding to a fixed FAT coloring (Section \ref{sec:removing}), then proceed to examine the effect of taking the complement of a graph (Section \ref{sec:comp}), and finally investigate how FAT colorings of two graphs $G_1$ and $G_2$ can be combined via operations defined on $G_1$ and $G_2$ (Section \ref{sec:products}).

\section{Background}\label{sec:background}

In this section, we recall the definitions, notations, and fundamental results concerning FAT colorings and spectral graph theory that will be used throughout the paper. To ensure that the exposition is self-contained, we restate the necessary auxiliary results from \cite{beersmulasFAT} and include their proofs.

\subsection{FAT colorings}
We fix a simple graph $G = (V, E)$ on $N$ vertices. In particular, we assume that $G$ is undirected and unweighted, without multi-edges and without loops.\newline

Given a vertex $v\in V$ and a subset $S\subseteq V$, we let
$$
e(v,S) := |\{w\in S : v\sim w\}|
$$
be the number of neighbors of $v$ contained in $S$. Given two (not necessarily distinct) subsets $S,T\subseteq V$, we let
$$
e(S,T) := |\{\{u,v\}\in E : u\in S, v\in T\}|
$$
be the number of edges with one endpoint in $S$ and the other in $T$.\newline 

Recall that a \emph{$k$-coloring} of $G$ is a function $c:V\to \{1,\ldots,k\}$, and for each $i\in\{1,\ldots,k\}$, the set $V_i := c^{-1}(i)$ is called the \emph{coloring class} of color $i$. Moreover, the coloring is \emph{proper} if $v\sim w$ implies $c(v)\neq c(w)$. The \emph{chromatic number} (or \emph{coloring number}) $\chi=\chi(G)$ is the smallest $k$ for which there exists a proper $k$-coloring of $G$.\newline 

We now recall the definitions of FAT coloring and FAT chromatic number from \cite{beersmulasFAT}.

\begin{definition}\label{def:FAT}
    A vertex $k$-coloring with coloring classes $V_1,\ldots,V_k$ is called \emph{Fair and Tolerant} (FAT) if there exists a parameter $\alpha\in [0,1]$ such that
    \[
    e(v,V_i)=
    \begin{cases}
        \alpha \deg v, & \text{if } v\notin V_i,\\[2mm]
        \beta \deg v, & \text{if } v\in V_i,
    \end{cases}
    \]
    where $\beta := 1 - (k-1)\alpha$.  The  \emph{FAT chromatic number} of a graph $G$, denoted by $\chi^{\textrm{FAT}}(G)$ or simply $\chi^{\textrm{FAT}}$, is the largest integer $k$ such that  $G$ admits a FAT $k$-coloring. 
\end{definition}

\begin{remark}
     In Fair and Tolerant colorings, \emph{fairness} reflects the fact that every vertex assigns the same fraction $\alpha$ of its neighbors to each of the other $k-1$ coloring classes, while \emph{tolerance} reflects the fact that a remaining fraction $\beta$ of its neighbors are allowed to share its own color (Figure \ref{fig:FATillustration}). Moreover, the reason why the FAT chromatic number is defined as the largest $k$ for which $G$ admits a FAT $k$-coloring is that every graph admits a FAT $1$-coloring in which all vertices belong to the same coloring class. 
\end{remark}

\begin{figure}[h]
    \centering
    \includegraphics[width=0.5\linewidth]{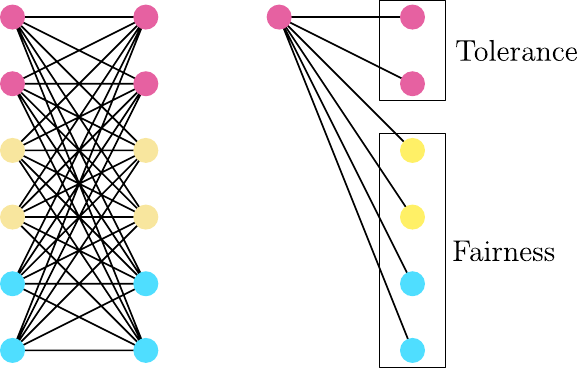}
    \caption{An illustration of the concepts of fairness and tolerance.}
    \label{fig:FATillustration}
\end{figure}

\begin{remark}
    If a vertex $v$ is \emph{isolated}, i.e., if its degree is equal to $0$, then the degree condition in Definition~\ref{def:FAT} is trivially satisfied for every coloring class $V_i$ and for all choices of $\alpha$ and $\beta$. Hence, an isolated vertex may be assigned an arbitrary color.
\end{remark}

\begin{example}\label{ex:completegraphs}
    Consider the complete graph $K_N$ on $N$ vertices. Assigning a distinct color to each vertex yields a FAT coloring with parameters
    \[
    \beta = 0, \qquad \alpha = \frac{1}{N-1}.
    \]
    Consequently,
    \[
    \chi^{\mathrm{FAT}}(K_N) = \chi(K_N) = N.
    \]
    One can easily check that all remaining FAT colorings of $K_N$ can be obtained by merging the coloring classes into groups of size $N/l$, for each divisor $l \mid N$.
\end{example}

By Example \ref{ex:completegraphs}, for every FAT coloring of a complete graph $K_N$, all coloring classes have the same cardinality. In \cite{beersmulasFAT}, it was shown that this phenomenon extends to all regular graphs:

\begin{theorem}[Theorem 2.12 in \cite{beersmulasFAT}]\label{thm:2.12}
    Let $G$ be a connected regular graph on $N$ vertices. For any FAT $k$-coloring of $G$, all coloring classes have the same size. In particular, each class has size $N/k$, and therefore $k$ divides $N$.
\end{theorem}

For completeness, we reproduce the proof, following the argument in \cite{beersmulasFAT}.

\begin{proof}
Let $G$ be a $d$-regular graph, and fix a FAT $k$-coloring with classes $V_1,\ldots,V_k$. If $k=1$, then there is only one coloring class, hence the claim is trivial. If $k>1$, then $\alpha\neq 0$ since $G$ is connected. In this case, for $i\neq j$ we have
$$
e(V_i,V_j)=\sum_{v\in V_i} e(v,V_j)=|V_i| \cdot \alpha \cdot d
$$ and
$$
e(V_i,V_j)=\sum_{w\in V_j} e(w,V_i)=|V_j| \cdot \alpha \cdot d.
$$
Hence, $$
|V_i| \cdot \alpha \cdot d=|V_j| \cdot \alpha \cdot d.
$$ Since $\alpha \neq 0$, this implies that $|V_i|=|V_j|$.  Therefore, each class has size $N/k$, and $k|N$. 
\end{proof}

\subsection{Spectral properties}

We now recall the main notions from spectral graph theory that shall be required in subsequent sections. We fix again a simple graph $G=(V,E)$ with vertices $v_1,\ldots,v_N$.

\begin{definition}
The \emph{adjacency matrix} of $G$ is the $N\times N$ matrix $A:=A(G)$ with entries
\[
A_{ij} :=
\begin{cases}
1, & \text{if } v_i\sim v_j,\\
0, & \text{otherwise.}
\end{cases}
\]
The \emph{degree matrix} of $G$ is the diagonal matrix
\[
D:=D(G)=\mathrm{diag}(\deg v_1,\ldots,\deg v_N).
\]
The \emph{normalized Laplacian} of $G$ is given by
\[
L:=L(G):=I-D^{-1}A.
\]
In case a vertex $v_i$ is isolated, we let $D_{i,i}^{-1}\coloneqq 0$, following the convention in Chung \cite{chung}.
\end{definition}

The normalized Laplacian $L$ has $N$ real, non-negative eigenvalues, all contained in the interval $[0,2]$ \cite{chung}. We denote them by
\[
0=\lambda_1\leq \lambda_2 \leq \cdots \leq \lambda_N \leq 2.
\]

Such eigenvalues reflect several structural features of the graph \cite{chung,BrouwerHaemers2012,spectra26}. In particular, the multiplicity of the eigenvalue $0$ equals the number of connected components of $G$, and when $G$ is connected, the second eigenvalue $\lambda_2$ is related to the Cheeger constant. Moreover, the largest eigenvalue $\lambda_N$ characterizes two extremal cases: it equals $2$ precisely when at least one connected component of $G$ is bipartite, and it equals $N/(N-1)$ exactly when $G$ is complete. \newline

\begin{remark}
    It is sometimes convenient to regard the normalized Laplacian $L$ as a linear operator on the space $C(V)$ of functions $f:V\to\mathbb{R}$, given by
    \[
    Lf(v) = f(v) - \frac{1}{\deg v}\sum_{w\sim v} f(w), \quad v\in V,
    \]
    which is equivalent to the matrix formulation above. In particular, a function $f$ is an eigenfunction with eigenvalue $\lambda$ if and only if $Lf=\lambda f$, that is,
    \begin{equation}\label{eq:eigenpair}
        \lambda f(v) = f(v) - \frac{1}{\deg v}\sum_{w\sim v} f(w), \quad v\in V.
    \end{equation}
    Given $f,g\in C(V)$, we also let
	\begin{equation}\label{eq:innerproduct}
	\langle f,g\rangle:=\sum_{v\in V}\deg v\cdot f(v)\cdot g(v).
	\end{equation}
    It is straightforward to check that $L$ is self-adjoint with respect to the inner product $	\langle \cdot,\cdot \rangle$, i.e.,
	\begin{equation*}
	\langle Lf,g\rangle=\langle f,Lg\rangle \quad \forall f,g\in C(V).
	\end{equation*}
\end{remark}

We now introduce a family of functions associated with ordered pairs of coloring classes, needed in Theorem \ref{thm:eigenfunctions} below.

\begin{definition}\label{def:f_ij}
   Fix a vertex $k$-coloring of $G$ with coloring classes $V_1,\ldots,V_k$. Given two distinct indices $i,j\in\{1,\ldots,k\}$, we let
    \[
    f_{ij}(v) \coloneqq \begin{cases}
        1, &\text{ if } v\in V_i,\\
        -1, &\text{ if } v\in V_j,\\
        0, &\text{ otherwise.}
    \end{cases}
    \]
\end{definition}

The following theorem provides an equivalent spectral characterization of FAT colorings and will be used throughout the paper. Although one implication follows directly from the proof of Theorem 3.3 in \cite{beersmulasFAT}, we include a complete proof here for completeness.\\

\begin{theorem}\label{thm:eigenfunctions}
    Let $G$ be a connected graph, and let $V_1,\ldots,V_k$ be a partition of $V(G)$. Then, the partition corresponds to a FAT coloring with parameter $\alpha$ if and only if, for every pair of distinct indices $i,j$, the functions $f_{ij}$ from Definition \ref{def:f_ij} are eigenfunctions of $L$. Moreover, in this case, all such eigenfunctions have the same eigenvalue $k\alpha$. 
\end{theorem}

\begin{proof}

The ``only if'' implication follows exactly as in the proof of Theorem 3.3 in \cite{beersmulasFAT}. For completeness, we reproduce the argument here.

Assume that $V_1,\ldots,V_k$ are the coloring classes of a FAT coloring with parameter $\alpha$. Observe that, by Definition \ref{def:FAT},
$$
1-\beta+\alpha = 1-1+ (k-1)\alpha + \alpha = k\alpha.
$$

Hence, to prove the ``only if'' implication, we need to show that \eqref{eq:eigenpair} holds for all $v\in V$, with $f=f_{ij}$ and $\lambda:=1-\beta+\alpha$.

To this end, we fix the indices $i$ and $j$, and for $v\in V$ we distinguish the following three cases.

\begin{itemize}

\item Case 1: $v\in V_i$. In this case, $f_{ij}(v)=1$. Moreover, among the neighbors of $v$, a fraction $\beta$ lies in $V_i$, and a fraction $\alpha$ lies in $V_j$. Therefore,
$$
f_{ij}(v)-\frac{1}{\deg v}\sum_{w\sim v}f_{ij}(w)=1-\frac{1}{\deg v}\left(\beta \deg v- \alpha \deg v \right)=1-\beta + \alpha = \lambda = \lambda \cdot f_{ij}(v).
$$
    
    \item Case 2: $v\in V_j$. In this case, $f_{ij}(v)=-1$. Moreover, among the neighbors of $v$, a fraction $\alpha$ lies in $V_i$, and a fraction $\beta$ lies in $V_j$. Therefore,
$$
f_{ij}(v)-\frac{1}{\deg v}\sum_{w\sim v}f_{ij}(w)=-1-\frac{1}{\deg v}\left(\alpha \deg v- \beta \deg v \right)=-1+\beta - \alpha = -\lambda = \lambda \cdot f_{ij}(v).
$$

    \item Case 3: $v\notin V_i\cup V_j$. In this case, $f_{ij}(v)=0$. Moreover, among the neighbors of $v$, a fraction $\alpha$ lies in $V_i$, and a fraction $\alpha$ lies in $V_j$. Therefore,
$$
f_{ij}(v)-\frac{1}{\deg v}\sum_{w\sim v}f_{ij}(w)=-\frac{1}{\deg v}\left(\alpha \deg v- \alpha \deg v \right)=0=\lambda \cdot f_{ij}(v).
$$

\end{itemize}

This proves the first implication. Conversely, assume that the functions $f_{ij}$ are eigenfunctions of $L$ with eigenvalue $k\alpha$. Then, for each pair of distinct indices $i,j$, and for each $v\in V$,

\begin{equation}\label{eq:fij-ef}
    f_{ij}(v)-\frac{1}{\deg v}\sum_{w\sim v}f_{ij}(w)=k\alpha \cdot f_{ij}(v).
\end{equation}

If $v\in V_i$, by Definition \ref{def:f_ij}, \eqref{eq:fij-ef} becomes

$$
1-\frac{1}{\deg v}\cdot \Bigl(e(v,V_i)-e(v,V_j)\Bigr)=k\alpha.
$$

Hence,

$$
(1-k\alpha)\cdot \deg v = e(v,V_i)-e(v,V_j),
$$ implying that

\begin{equation}\label{eq:e(v,Vj)}
    e(v,V_j)=e(v,V_i)-(1-k\alpha)\cdot \deg v.
\end{equation}

 In particular, $e(v,V_j)$ does not depend on $j$, for each fixed $i$, each $v\in V_i$, and each $j\neq i$. From the last equality, we can also infer that

\begin{align*}
    \deg v&=e(v,V_i)+\sum_{j\neq i}e(v,V_j)\\ &=e(v,V_i)+(k-1)\cdot \Bigl(e(v,V_i)-(1-k\alpha)\cdot \deg v\Bigr)\\
    &= k\cdot e(v,V_i)-(k-1)(1-k\alpha)\deg v.
\end{align*}

Therefore,

$$
k\cdot e(v,V_i)=\deg v+(k-1)(1-k\alpha)\deg v=\deg v\Bigl(1+(k-1)(1-k\alpha)\Bigr)=\deg v(k-k^2\alpha+k\alpha),
$$ implying that

\begin{equation}\label{eq:e(v,Vi)}
   e(v,V_i)=\deg v\cdot (1-k\alpha+\alpha). 
\end{equation}

By putting together \eqref{eq:e(v,Vj)} and \eqref{eq:e(v,Vi)}, we obtain that

\begin{equation}\label{eq:e(v,Vj)2}
    e(v,V_j)=e(v,V_i)-(1-k\alpha)\cdot \deg v=\deg v\cdot (1-k\alpha+\alpha)-(1-k\alpha)\cdot \deg v=\alpha \deg v.
\end{equation}

By \eqref{eq:e(v,Vi)} and \eqref{eq:e(v,Vj)2}, the partition $V_1,\ldots,V_k$ corresponds to a FAT coloring with parameter $\alpha$. This proves the claim.
    
\end{proof}

We conclude this section by fixing the following notation.

\begin{definition}
    Let $G_1$ and $G_2$ be graphs. Whenever a function $f^{(1)}$ or $f^{(2)}$ is considered, it is implicitly assumed that
    \[
    f^{(i)} \colon V(G_i) \to \mathbb{R} \quad \text{for } i = 1,2.
    \]
    Furthermore, a function $f^{(1,2)}$ is considered, it is implicitly assumed that its domain is the vertex set of a graph product of $G_1$ and $G_2$.
\end{definition}

Finally, if an eigenvalue $\lambda$ occurs with multiplicity $k$, we write $\lambda^{(k)}$.

\section{Regular Turán graphs}\label{section:Turan}
In \cite{beersmulasFAT}, the FAT chromatic number of regular Turán graphs was determined via a combinatorial argument. In this work, we present an alternative and more concise proof of the same result, based on spectral techniques.

The motivation for providing an alternative proof is twofold. First, it highlights the effectiveness of spectral methods in the study of FAT colorings. Second, the argument we develop here arises as a special case of a more general result for complete multipartite graphs, which shall be addressed in Section \ref{sec:multipartite}.

We begin by recalling the relevant definitions and fixing notation, following the conventions in \cite{beersmulasFAT}.

\begin{definition}
    A \emph{complete multipartite graph} is a graph whose vertex set is partitioned into pairwise disjoint independent sets $P_1,\dots,P_t$ (referred to as \emph{parts} or \emph{Turán parts}), with an edge between every pair of vertices belonging to distinct parts. If the parts have cardinalities $n_1,\dots,n_t$, we denote such a graph by $K_{n_1,\dots,n_t}$.
\end{definition}

\begin{figure}[h]
    \centering
    \includegraphics[scale=0.66]{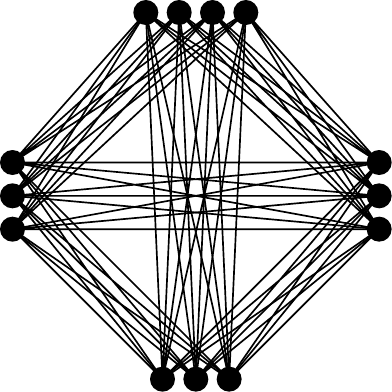}
    \caption{The Turán graph $T(13,4)$.}
    \label{fig:Turan}
\end{figure}

\begin{definition}
    The \emph{Turán graph} $T(N,t)$ (see Figure~\ref{fig:Turan}) is the complete $t$-partite graph on $N$ vertices whose parts are as equal in size as possible; that is, each part has either $\lfloor N/t \rfloor$ or $\lceil N/t \rceil$ vertices.
\end{definition}

\begin{definition}
    A \emph{regular Turán graph} is a Turán graph $T(N,t)$ in which $t$ divides $N$, so that all parts have the same size $N/t$. In this case, the graph is $d$-regular, with degree
    \[
        d = N - \frac{N}{t}.
    \]
\end{definition}

We additionally recall Theorem 4.4 from \cite{beersmulasFAT}, for which we shall provide an alternative proof.

\begin{theorem}\label{thm:Turan}[Theorem 4.4 from \cite{beersmulas2024}]
Let $T(N,t)$ denote the regular Turán graph on $N$ vertices with $t$ equal parts $P_1,\ldots,P_t$ of size $N/t$. Every FAT coloring of $T(N,t)$ belongs to one of the following two classes (see Figure \ref{fig:T(12,4)bm}):
\begin{enumerate}
    \item \emph{Balanced case.} Each Turán part contains all colors, each occurring with the same multiplicity.
    \item \emph{Monochromatic case.} Each Turán part contains exactly one color, and each color appears in the same number of Turán parts.
\end{enumerate}
In particular,
\[
\chi^{\textrm{FAT}}(T(N,t))=\max\{t,N/t\}.
\]
\end{theorem}

\begin{figure}[h]
    \centering
    \includegraphics[scale=0.66]{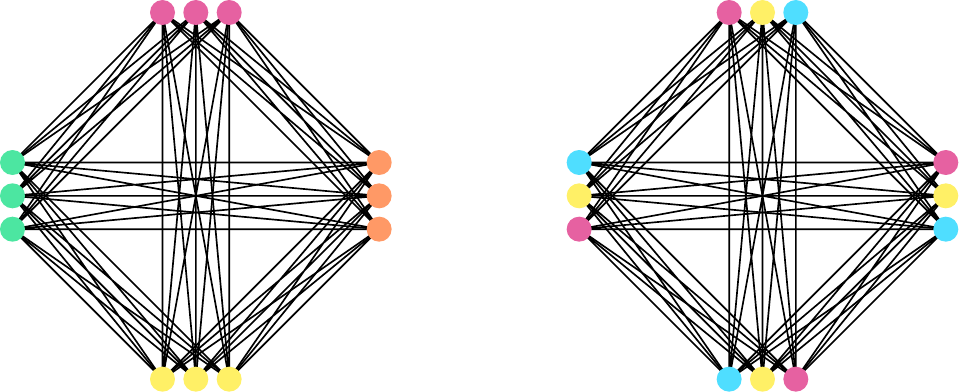}
    \caption{A monochromatic (left) and balanced (right) FAT coloring of $T(12,4)$.}
    \label{fig:T(12,4)bm}
\end{figure}

The proof of Theorem \ref{thm:Turan} that we present here is based on the spectral properties of the normalized Laplacian, and it provides a shorter, spectral alternative to the combinatorial proof given in \cite{beersmulasFAT}. Before proceeding, we fix the regular Turán graph $T(N,t)$ on $N$ vertices with $t$ equal parts $P_1,\ldots,P_t$, each of size $N/t$, and we introduce some auxiliary notation.

\begin{definition}\label{def:f_vw}
    Let $v,w\in V$ be distinct vertices. We let
    \[
    f_{v,w}\colon V\to\R
    \]
    be such that
    \begin{equation*}
    f_{v,w}(u) \coloneqq \begin{cases}
        1, &\text{if } u = v,\\
        -1, &\text{if } u = w,\\
        0, &\text{otherwise.}
    \end{cases}
    \end{equation*}
\end{definition}

Analogously to the functions $f_{ij}$ from Definition \ref{def:f_ij}, which are defined in terms of the coloring classes, we also introduce  functions $g_{ij}$ defined with respect to the Turán parts, as follows.

\begin{definition}\label{def:g_ij}
    For two distinct indices $i,j\in\{1,\ldots,t\}$, we let $g_{ij}\colon V\bigl(T(N,t)\bigr)\to\R$ be such that
    \[
    g_{ij}(v) \coloneqq \begin{cases}
        1, &\text{if } v\in P_i,\\
        -1, &\text{if } v\in P_j,\\
        0, &\text{otherwise.}
    \end{cases}
    \]
\end{definition}

Having established the above preliminaries, we are now in a position to present the spectral proof of Theorem \ref{thm:Turan}.

\begin{proof}[Spectral proof of Theorem \ref{thm:Turan}]
    The spectrum of regular Turán graphs is given by
    \[
    \biggl\{\frac{t}{t-1}^{(t-1)},1^{(N-t)},0^{(1)}\biggr\},
    \]
    as established in \cite{SunDas2020.2} by combining Theorems 3.1 and 3.5.

    The eigenspaces $E_{t/(t-1)}, E_1$ and $E_0$ associated with their corresponding eigenvalues, listed in non-increasing order, can be described in terms of eigenfunctions:
    \begin{align*}
        E_{t/(t-1)} &= \bigl\langle g_{1i}\colon i=2,\ldots,t\bigr\rangle, \\
        E_1 &= \bigl\langle f_{v,w}\colon v,w\in P_\ell \text{ for some } \ell=1,\ldots,t\bigr\rangle,\\
        E_0 &= \langle \boldsymbol 1 \rangle,
    \end{align*} where $\boldsymbol 1: V \to \mathbb{R}$ denotes the function that takes the value $1$ on every vertex.

    Now, let $V_1,\ldots,V_k$ denote the coloring classes of a fixed FAT coloring. By Theorem \ref{thm:eigenfunctions}, these FAT coloring classes induce eigenfunctions $f_{ij}$ with common eigenvalue $k\alpha$. Consequently, our FAT coloring must correspond to one of the three eigenvalues listed above. We consider the following three cases.

\begin{itemize}
    \item Case $i$:
    The functions $f_{ij}$ have eigenvalue $0$. In this case, $k=1$, implying that the FAT coloring is trivial, and therefore it falls into the monochromatic case described in the statement of Theorem \ref{thm:Turan}.
    
   \item Case $ii$:
     The functions $f_{ij}$ have eigenvalue $1$, that is, $f_{ij}\in E_1$. Since
     \[
     E_1=\left\{f\colon V\to\R : \sum_{v\in P_k}f(v)=0 \text{ for all } k=1,\ldots,t\right\},
     \]
     it follows that for any distinct $i,j\in\{1,\ldots,k\}$ and any $\ell\in\{1,\ldots,t\}$,
     \[
     \bigl|V_i\cap P_\ell\bigr|-\bigl|V_j\cap P_\ell\bigr|
       =\sum_{\substack{v\in P_\ell\\ v\in V_i}}f_{ij}(v)
        + \sum_{\substack{v\in P_\ell\\ v\in V_j}}f_{ij}(v)
       =\sum_{v\in P_\ell}f_{ij}(v)
       =0.
     \]
     Hence, for each coloring class $V_i$ and each Turán part $P_\ell$, the cardinality $\lvert V_i \cap P_\ell\rvert$ equals $|P_{\ell}|/k=N/(tk)$, and is therefore independent of the choice of $V_i$ and $P_\ell$. Equivalently, we are in the balanced case of Theorem \ref{thm:Turan}: every Turán part contains all colors, and each color occurs with the same multiplicity in every part.
     
    \item  Case $iii$:
     The functions $f_{ij}$ have eigenvalue $t/(t-1)$, i.e., $f_{ij}\in E_{t/(t-1)}$. From the above description of $E_{t/(t-1)}$, it follows directly that each Turán part contains exactly one color. 
Moreover, the FAT coloring must use at least two colors (since Case $i$ does not apply).
 Hence, every coloring class is the union of the same number of Turán parts, which again places us in the monochromatic case of Theorem \ref{thm:Turan}.
\end{itemize}
\end{proof}

The spectral proof of Theorem \ref{thm:Turan} illustrates the role of Theorem \ref{thm:eigenfunctions}. The eigenfunctions of $G$ encode essential structural information about its edge distribution, so that the analysis of FAT colorings of $G$ can be reduced to the study of its spectrum and corresponding eigenspaces. In the subsequent sections, we shall encounter further instances where the spectrum and eigenfunctions provide useful tools for investigating FAT colorings.

\section{Complete multipartite graphs}\label{sec:multipartite}
In this section, we extend the spectral proof of Theorem \ref{thm:Turan} from the previous section to the setting of arbitrary complete multipartite graphs.

We adopt the notation
\begin{equation}\label{eq:compmult}
    K_{\underbrace{n_1, \ldots, n_1}_{\theta_1}, \ldots, \underbrace{n_p, \ldots, n_p}_{\theta_p}}
\end{equation}
to denote the complete multipartite graph whose partition classes are
\begin{equation}\label{eq:partclasses}
    P^1_1,\ldots,P^1_{\theta_1},\ldots,P^p_{1},\ldots,P^p_{\theta_p}
\end{equation}
with cardinalities $\bigl|P^i_k\bigr| = n_i$,
ordered so that $n_i> n_j$ whenever $i> j$.
The (normalized Laplacian) spectra of such graphs were partially determined by Sun and Das (2020) \cite{SunDas2020.2}, who established the following result.

\begin{lemma}[Lemma 3.4 in \cite{SunDas2020.2}]\label{lem:3.4sundas}
    Let 
    \[
    K_{\underbrace{n_1, \ldots, n_1}_{\theta_1}, \ldots, \underbrace{n_p, \ldots, n_p}_{\theta_p}}
    \]
    be a complete $k$-partite graph on $N$ vertices as in \eqref{eq:compmult}. Its normalized Laplacian spectrum is
    \begin{equation}
        \biggl\{ \frac{N}{N-n_1}^{(\theta_1-1)},  \ldots, \frac{N}{N-n_p}^{(\theta_p-1)}, 1^{\bigl(N-\sum_{i=1}^p\theta_i\bigr)}, x_1^{(1)},\ldots,x_{p-1}^{(1)},0^{(1)}\biggr\},
    \end{equation}
    where the eigenvalues $x_i$ satisfy
    \[
    \frac{N}{N-n_{i+1}}<x_i<\frac{N}{N-n_i} \quad \text{for } i=1,\ldots,p-1.
    \]
\end{lemma}

We now proceed to determine certain eigenfunctions of the graph
\[
K_{\underbrace{n_1, \ldots, n_1}_{\theta_1}, \ldots, \underbrace{n_p, \ldots, n_p}_{\theta_p}}.
\]
To this end, we first formulate and prove the following lemma.

\begin{lemma}\label{lem:plusminus}
Let $V_+,V_-\subset V$ be disjoint subsets of the vertex set $V$, and define the function $f_{\pm}:V\to\mathbb{R}$ by
\[
    f_{\pm}(v) := \begin{cases}
        1 &\text{ if } v\in V_+,\\
        -1 &\text{ if } v\in V_-,\\
        0 &\text{ otherwise.}
    \end{cases}
\]
Then, $f_{\pm}$ is an eigenfunction of the Laplacian $L$ with eigenvalue $\lambda$ if and only if the following two conditions are satisfied:
\begin{enumerate}
    \item For every vertex $v_0\notin V_+\cup V_-$,
    \[
        e(v_0,V_+) = e(v_0,V_-),
    \]
    i.e., $v_0$ has the same number of neighbors in $V_+$ and in $V_-$.
    \item For every $v_-\in V_-$ and $v_+\in V_+$,
    \[
        \lambda-1
        = \frac{e\bigl(v_-,V_+\bigr)-e\bigl(v_-,V_-\bigr)}{\deg v_-}
        = \frac{e\bigl(v_+,V_-\bigr)-e\bigl(v_+,V_+\bigr)}{\deg v_+}.
    \]
    In particular, if $V_-$ and $V_+$ are independent sets, the identity above simplifies to
    \[
        \lambda-1 = \frac{e\bigl(v_-,V_+\bigr)}{\deg v_-}
        = \frac{e\bigl(v_+,V_-\bigr)}{\deg v_+}.
    \]
\end{enumerate}
\end{lemma}
\begin{proof}
    We begin by recalling that, by \eqref{eq:eigenpair}, the pair $(f_{\pm},\lambda)$ is an eigenpair if and only if, for all $v\in V$,
    \[
    (\lambda-1)f(v) = -\frac1{\deg v}\sum_{w\sim v}f(w).
    \]
    Consequently, for any choice of $v_0\notin V_+\cup V_-$, $v_-\in V_-$, and $v_+\in V_+$, the pair $(f_{\pm},\lambda)$ is an eigenpair if and only if the following three identities are satisfied:

    \begin{align*}
        0 &= (\lambda-1)f(v_0) = \frac{e\bigl(v_0,V_+\bigr) - e\bigl(v_0,V_-\bigr)}{\deg v_0},\\
        \lambda-1 &= -(\lambda-1)f\bigl(v_-\bigr) = \frac{e\bigl(v_-,V_+\bigr) - e\bigl(v_-,V_-\bigr)}{\deg v_-},\\
        \lambda-1 &= (\lambda-1)f\bigl(v_+\bigr) = \frac{e\bigl(v_+,V_-\bigr) - e\bigl(v_+,V_+\bigr)}{\deg v_+}.
    \end{align*}
    The identities above establish the desired equalities and thereby conclude the proof.
\end{proof}
As an immediate consequence of Lemma \ref{lem:plusminus}, we obtain a characterization of a subclass of eigenfunctions of the graph $K_{\underbrace{n_1, \ldots, n_1}_{\theta_1}, \ldots, \underbrace{n_p, \ldots, n_p}_{\theta_p}}$.

\begin{corollary}\label{cor:eigencompmult}
    The eigenfunctions of $K_{\underbrace{n_1, \ldots, n_1}_{\theta_1}, \ldots, \underbrace{n_p, \ldots, n_p}_{\theta_p}}$ corresponding to the eigenvalue $N/\bigl(N-n_i\bigr)$ are precisely the linear combinations of the functions
        \begin{equation*}
        g^i_{jk}(v)\coloneqq \begin{cases}
            1, &\text{ if } v\in P^i_j,\\
            -1, &\text{ if } v\in P^i_k,\\
            0, &\text{ otherwise.}
        \end{cases}
        \end{equation*}
    Furthermore, the eigenfunctions associated with the eigenvalue $1$ are exactly the linear combinations of the functions $f_{vw}$ from \eqref{def:f_vw}, where $v$ and $w$ belong to the same partition class.
\end{corollary}
We are now prepared to formulate and prove the following generalization of Theorem \ref{thm:Turan}, which yields a solution to Question 7 from \cite{beersmulasFAT}: “Can Theorem \ref{thm:Turan}, which characterizes the FAT chromatic number of regular Turán graphs, be extended to all complete multipartite graphs?”.

\begin{theorem}\label{thm:compmult}
    Let $G \coloneqq K_{\underbrace{n_1,\ldots,n_1}_{\theta_1},\ldots,\underbrace{n_p,\ldots,n_p}_{\theta_p}}$ be a complete multipartite graph as in~\eqref{eq:compmult}. Then, the FAT coloring number of $G$ is given by (see Figure~\ref{fig:FATcompmult})
    \[
        \chi^{FAT}(G) =
        \begin{cases}
            \max\{\theta_1,n_1\}, & \text{if } p=1,\\[2mm]
           2, & \text{if } p = 2, \, \gcd\{n_1,n_2\} = 1, \text{ and }\\ &  \theta
            _1(\theta_1-1)n_1^2 = \theta_2(\theta_2-1)n_2^2,\\[2mm]
            \gcd\{n_i : i=1,\ldots,p\}, & \text{otherwise.}
        \end{cases}
    \]
\end{theorem}

\begin{figure}[h]
    \centering
    \includegraphics[scale=0.66]{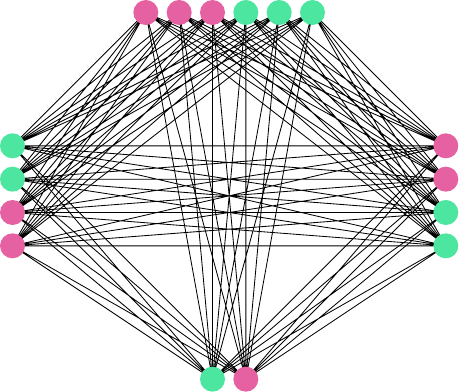}
    \caption{A FAT coloring with $\chi^{FAT}=2$ colors of the complete multipartite graph $K_{6,4,4,2}$.}
    \label{fig:FATcompmult}
\end{figure}

\begin{proof}
    We distinguish different cases according to the value of $p$. If $p=1$, then $G$ is a regular Turán graph, and the assertion follows directly from Theorem~\ref{thm:Turan}.
    
    Now assume that $p>1$. Note that the eigenvalues $N/(N-n_i)$ for $i=1,\ldots,p$ cannot give rise to a FAT coloring, since the associated eigenspace does not have support equal to $V(G)$, by Corollary \ref{cor:eigencompmult}. Furthermore, the eigenvalue $0$ gives rise exclusively to the trivial FAT coloring. Hence, it is sufficient to restrict our attention to the eigenvalues $1$ and $x_i$.\\
    
    First suppose that there exists a FAT $k$-coloring of $G$ corresponding to the eigenvalue $1$, and denote its coloring classes by $V_1,\ldots,V_k$.
    By combining Theorem~\ref{thm:eigenfunctions} with Corollary~\ref{cor:eigencompmult}, and arguing in complete analogy with Case~$ii$ in the spectral proof of Theorem~\ref{thm:Turan}, we obtain, for every coloring class $V_i$ and each partition class $P_j^\ell$ (cf.\ Equation \eqref{eq:partclasses}),
    \[
        \bigl|P_j^\ell \cap V_i\bigr| = \frac{\bigl|P_j^\ell\bigr|}{k}.
    \]
    Consequently, $k$ must be a common divisor of the sizes of all partition classes. Since $\lvert P_j^l\rvert = n_i$ for the corresponding index $i$, it follows that the FAT coloring number of $G$ is equal to the largest integer $k$ dividing $n_i$ for all $i=1,\ldots,p$, that is,
    \[
        \chi^{FAT}(G) = \gcd\{n_i : i=1,\ldots,p\}.
    \]
  
    In this case, a FAT $\chi^{FAT}(G)$-coloring is obtained by distributing the $\chi^{FAT}(G)$ colors within each partition class $P_j^{\ell}$ so that exactly $|P_j^{\ell}|/\chi^{FAT}(G)$ vertices of $P_j^{\ell}$ receive each color (see, for example, Figure~\ref{fig:FATcompmult}).\\

    Finally, assume that there exists a FAT $k$-coloring of $G$ corresponding to one of the eigenvalues $x_i$ (cf.\ Lemma \ref{lem:3.4sundas}). In this case, since the multiplicity of the eigenvalue equals $1$, the corresponding FAT coloring must consist of two coloring classes $V_1$ and $V_2$, and a corresponding eigenvector must be of the form $f_{12}$ as in Definition \ref{def:f_ij}.
    
    Furthermore, since $L$ is self-adjoint with respect to the inner product $\langle \cdot,\cdot\rangle$ from Equation \eqref{eq:innerproduct},
    the eigenvector $f_{12}$ must be orthogonal (with respect to the inner product in \eqref{eq:innerproduct}) to the eigenvectors $f_{vw}$ from Definition \ref{def:f_vw} for $v,w$ in the same partition class, since $f_{vw}$ and $f_{12}$ are eigenvectors that correspond to different eigenvalues.

    If $v$ and $w$ belong to the same partition class, they must have the same degree. For $f_{vw}$ to be orthogonal to $f_{12}$, we must have
    \begin{align*}
        \sum_{u \in V}\deg u f_{12}(u)f_{vw}(u)&= 0 \\
        \deg v f_{12}(v) - \deg w f_{12}(w) &= 0 \\
        f_{12}(u) &= f_{12}(v)
    \end{align*}
    Hence, $f_{12}$ must be constant on any single partition class.\\

    Analogously, the eigenvector $f_{12}$ must be orthogonal (with respect to the inner product in \eqref{eq:innerproduct}) to the eigenvectors with eigenvalue $N/(N-n_j)\neq x_i$ for $k=1,\ldots,p$.

    Let now $j\in \{1,\ldots,p\}$ such that $\theta_j\geq 2$ (which is the case if and only if $N/(N-n_j)$ is an eigenvalue). Let also $P_k^j \neq P_{\ell}^j$ be two distinct partition classes of cardinality $n_j$, and consider the eigenfunction $g_{k \ell}^j$ as in Corollary \ref{cor:eigencompmult}. For $f_{12}$ to be orthogonal to $g_{k \ell}^j$, it must be the case that
    \begin{align*}
        &\sum_{u\in V}\deg u f_{12}(u)g_{k \ell}^j(u) = 0, \\
        &\sum_{u\in P_k^j} \deg u f_{12}(u) = \sum_{u\in P_\ell^j} \deg u f_{12}(u).
    \end{align*}
    Because $\deg u_1 = \deg u_2$ for $u_1\in P_k^j$ and $u_2\in P_\ell^j$, because we already established that $f_{12}$ must be constant on $P_k^j$ and on $P_\ell^j$, and because $|P_k^j| = |P_\ell^j| = \theta_j$,
    it follows that $f_{12}$ must have the same value on all classes $P_m^j$ for any fixed $j\in \{1,\ldots,p\}$.\\

    Now, let $i_1,i_2\in \{1,\ldots,p\}$ be such that $P_j^{i_1},P_k^{i_2}\subseteq V_1$. Let $v_1\in P_j^{i_1}$ and $v_2\in P_k^{i_2}$. Then,
    \[
    \deg v_1 = N - n_{i_1} \text{ and } \deg v_2 = N-n_{i_2}.
    \]
    If $V_1$ and $V_2$ induce a FAT coloring with parameter $\alpha$, we must have
    \[
    \frac{|V_2|}{N-n_{i_1}} = \alpha = \frac{|V_2|}{N-n_{i_2}}.
    \]
    Consequently, $n_{i_1} = n_{i_2}$, which means that, without loss of generality, $$V_1 = \bigcup_{j=1}^{\theta_1} P_j^1.$$ Similarly, one can derive that, without loss of generality,
    \[
    V_2 = \bigcup_{j=1}^{\theta_2} P_j^2.
    \]
    It follows that $p=2$. Furthermore, note that, for $v_i \in P_j^i$, we have
    \begin{align*}
        \deg v_1 &= \theta_2 \cdot n_2 + (\theta_1-1)\cdot n_1,\\
        \deg v_2 &= \theta_1 \cdot n_1 + (\theta_2 -1) \cdot n_2.
    \end{align*}
    Since a vertex in $V_1$ has $\theta_2\cdot n_2$ neighbors in $V_2$, and a vertex in $V_2$ has $\theta_1\cdot n_1$ neighbors in $V_1$, we see that
    \[
    \frac{\theta_2 \cdot n_2}{\theta_2 \cdot n_2 + (\theta_1-1)\cdot n_1} = \alpha = \frac{\theta_1\cdot n_1}{\theta_1 \cdot n_1 + (\theta_2-1)\cdot n_2}.
    \]
    Rearranging yields
    \begin{equation}\label{eq:x_i}
    \theta_1(\theta_1-1)n_1^2 = \theta_2(\theta_2-1)n_2^2.
    \end{equation}
    Hence, if $p=2$ and if Equation \eqref{eq:x_i} is satisfied, then $G$ admits a FAT coloring with $2$ colors. If $\gcd\{n_1,n_2\}>1$, $G$ also admits a FAT coloring with $\gcd\{n_1,n_2\}\geq 2$ colors. Hence, when $\gcd\{n_1,n_2\}=1$, one finds that $\chi^{FAT} = 2$ if and only if Equation \eqref{eq:x_i} is satisfied.
\end{proof}
One might wonder if there exist complete multipartite graphs with coloring classes of two different sizes $n_1$ and $n_2$, such that Equation \eqref{eq:x_i} is satisfied. We therefore end this section with the following remark concerning solutions of the equation.\\

\begin{remark}
    We are interested in solutions of Equation \eqref{eq:x_i} for $\theta_i,n_i\in \mathbb Z_{\geq1}$ and $i=1,2$, where $\gcd\{n_1,n_2\} = 1$. These variables correspond to the complete multipartite graph
    \[
    K_{\underbrace{n_1, \ldots, n_1}_{\theta_1}, \underbrace{n_2, \ldots, n_2}_{\theta_2}}.
    \]
    First of all, the equation 
    \[
    \theta_1(\theta_1-1)n_1^2 = \theta_2(\theta_2-1)n_2^2
    \]
    has trivial solutions $\theta_1=\theta_2 = 1$ and $n_1,n_2\in \mathbb Z_{\geq1}$ such that $\gcd\{n_1,n_2\}=1$. Those correspond to complete bipartite graphs. The corresponding FAT $2$-coloring is the bipartite coloring, with parameter $\alpha = 1$.

    A non-trivial solution to Equation \eqref{eq:x_i} with $\gcd\{n_1,n_2\}=1$ is given by
    \[
    (n_1,\theta_1) = (6,2) \quad \text{ and }\quad (n_2,\theta_2) = (1,9).
    \]
    The graph $K_{6,6,1,1,1,1,1,1,1,1,1}$ and its FAT $2$-coloring can be found in Figure \ref{fig:FATcomplicated}. The corresponding parameter of its unique FAT $2$-coloring is
    \[
    \alpha = \frac{12}{20}=\frac{9}{15}=\frac35.
    \]
\end{remark}

\begin{figure}
        \centering
        \includegraphics[scale=0.66]{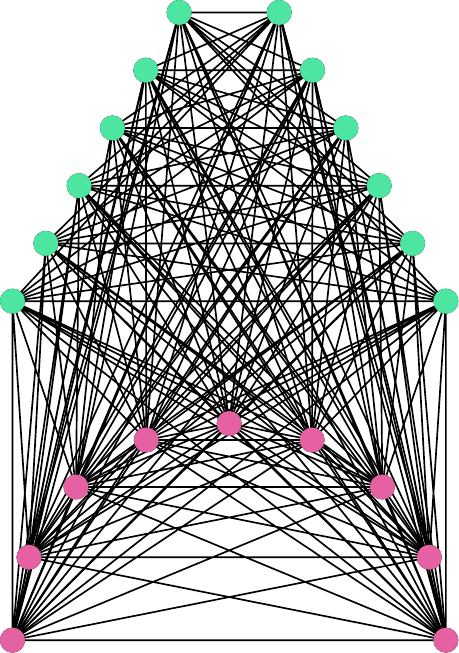}
        \caption{A FAT $2$-coloring of the graph $K_{6,6,1,1,1,1,1,1,1,1,1}$.}
        \label{fig:FATcomplicated}
    \end{figure}

\section{Removing a coloring class from a FAT coloring}\label{sec:removing}
In this section we show that, if a graph $G$ admits a FAT $k$-coloring, then it is possible to remove $i$ coloring classes from $G$ (see Figure~\ref{fig:removingcc}) for some $i \in \{1,\ldots,k-1\}$ such that the resulting graph remains FAT $(k-i)$-colorable. Furthermore, we establish a spectral property of the resulting graph.

\begin{figure}
    \centering
    \includegraphics[scale=0.66]{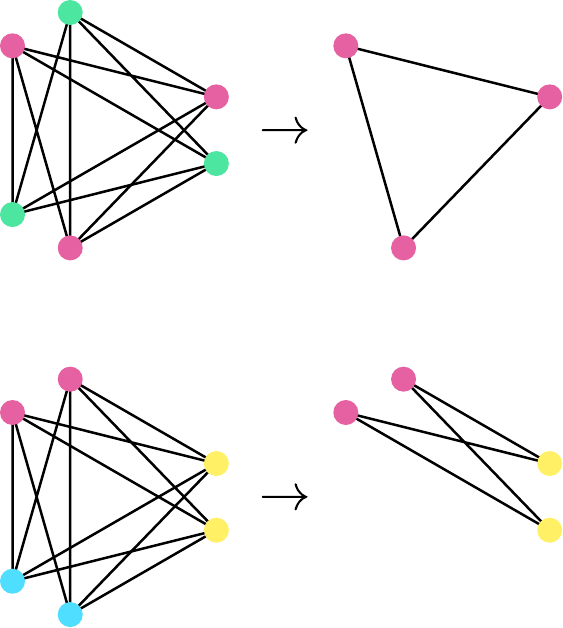}
    \caption{The Turán graph $T(6,3)$ before and after the removal of a coloring class corresponding to two distinct FAT colorings.}
    \label{fig:removingcc}
\end{figure}

\begin{theorem}\label{thm:removing}
    Let $V_1,\ldots,V_k$ denote the coloring classes of $G$ with respect to a FAT $k$-coloring with parameter $\alpha$, and let $I\subseteq \{1,\ldots,k\}$ be an index set of size $\lvert I\rvert\leq k-2$. Then, the induced subgraph $G\bigl(V\setminus\bigcup_{i\in I}V_i\bigr)$ admits a FAT $(k-\lvert I\rvert)$-coloring with parameter
    \[
        \frac{\alpha}{1-\alpha\lvert I\rvert}.
    \]
\end{theorem}
\begin{proof}
    Let \[
        G_{I}\coloneqq G\bigl(V\setminus\bigcup_{i\in I}V_i\bigr),
    \] and fix $v\in V(G_I)$.
    In the original graph $G$, the vertex $v$ has $\deg_G(v)$ neighbors. In the induced subgraph $G_I$, all neighbors of $v$ that lie in the coloring classes indexed by $I$ are removed. Since the given coloring is FAT with parameter $\alpha$, the vertex $v$ has exactly $\alpha\,\deg_G(v)$ neighbors in each coloring class $V_i$ for which $v\notin V_i$. In particular, for each $i\in I$, $v$ has $\alpha\,\deg_G(v)$ neighbors in $V_i$, and hence
    \[
        \deg_{G_I}(v)
        \;=\;
        \deg_G(v) - \alpha\lvert I\rvert\,\deg_G(v)
        \;=\;
        \bigl(1-\alpha\lvert I\rvert\bigr)\deg_G(v).
    \]
    Moreover, for any coloring class $V_i$ with $i\notin I$ and $v\notin V_i$, the vertex $v$ still has $\alpha\,\deg_G(v)$ neighbors in $V_i$. Rewriting in terms of $\deg_{G_I}(v)$ yields
    \[
        \alpha\,\deg_G(v)
        \;=\;
        \frac{\alpha}{1-\alpha\lvert I\rvert}\,\bigl(1-\alpha\lvert I\rvert\bigr)\deg_G(v)
        \;=\;
        \frac{\alpha}{1-\alpha\lvert I\rvert}\,\deg_{G_I}(v).
    \]
    Consequently, in $G_I$ the vertex $v$ has exactly
    \[
        \frac{\alpha}{1-\alpha\lvert I\rvert}\,\deg_{G_I}(v)
    \]
    neighbors in each remaining coloring class $V_i$ with $i\notin I$ and $v\notin V_i$. Hence, the sets $V_i$ for $i\in \{1,\ldots,k\}\setminus I$ form a FAT $(k-\lvert I\rvert)$-coloring of $G_I$ with parameter $\alpha/(1-\alpha\lvert I\rvert)$.
\end{proof}
We now consider the spectral viewpoint. We begin by establishing a general statement concerning eigenfunctions of graphs obtained by deleting certain FAT coloring classes from a graph that admits a FAT coloring. The result generalizes Proposition 3.14 from \cite{beersmulas2024}, which addresses only the case of proper FAT $k$-colorings (for which $\beta = 0$ and $\alpha = 1/(k-1)$).

\begin{proposition}\label{prop:removingef}
    Let $G$ be a graph, and let $V_1,\ldots,V_k$ denote the coloring classes with respect to a fixed FAT $k$-coloring.
    Furthermore, let $f\colon V(G)\to \R$ be an eigenfunction with corresponding eigenvalue $\lambda$, and suppose that $\supp(f)\subseteq \bigcup_{i\in I}V_i$ for some index set $I\subseteq \{1,\ldots,k\}$ with $|I|\geq 2$. Then the restriction $f|_{G\left(\bigcup_{i\in I}V_i\right)}$ is an eigenfunction of the induced subgraph $G\left(\bigcup_{i\in I}V_i\right)$ with eigenvalue
    \[
        1+\frac{\lambda-1}{1-|I|\alpha}.
    \]
\end{proposition}

\begin{proof}
    Set $G_I\coloneqq G\left(\bigcup_{i\in I}V_i\right)$, and let $f_I\coloneqq f|_{G_I}$ denote the restriction of $f$ to $G_I$. Moreover, let $v\in \bigcup_{i\in I}V_i$. By Equation \eqref{eq:eigenpair}, we obtain
    \begin{align*}
        \biggl(1-\biggl(1+\frac{\lambda-1}{1-|I|\alpha}\biggr)\biggr)f_I(v)
        &= \biggl(\frac{1-\lambda}{1-|I|\alpha}\biggr)f(v)\\
        &= \frac{1}{1-|I|\alpha}\cdot\frac{1}{\deg_G v}\sum_{\substack{w\in V(G):\\ \{w,v\}\in E(G)}} f(w)\\
        &= \frac{1}{\deg_{G_I}v}\sum_{\substack{w\in V(G_I):\\\{w,v\}\in E(G_I)}} f_I(w).
    \end{align*}
    The final equality follows from the fact that, for each vertex $v$, we remove exactly $\alpha |I| \deg_G v$ of its neighbors in the construction of the graph $G_I$, together with the assumption that $\supp(f)\subseteq \bigcup_{i\in I}V_i$. Consequently,
\[
    \deg_{G_I} v = (1 - |I|\alpha)\deg_G v.
\]
Hence, the function $f_I$ satisfies the eigenvalue equation on $G_I$ with corresponding eigenvalue
\[
    1 + \frac{\lambda - 1}{1 - |I|\alpha}.\qedhere
\]
\end{proof}

We conclude this section by presenting an alternative proof of Theorem \ref{thm:removing} based on spectral methods.

\begin{proof}[Spectral proof of Theorem \ref{thm:removing}]

Let $V_1,\ldots,V_k$ denote the coloring classes of $G$ with respect to a FAT $k$-coloring with parameter $\alpha$. Let $I \subseteq \{1,\ldots,k\}$ be a subset with size $\lvert I \rvert \leq k-2$.

For $i,j \in \{1,\ldots,k\} \setminus I$ with $i \neq j$, define $f_{ij}^{(I)} \coloneqq f_{ij}\rvert_{G_I}$, where $G_I$ is as in the proof of Theorem \ref{thm:removing}. Since $f_{ij}$ is an eigenfunction of $G$ with eigenvalue $k\alpha$ by Theorem \ref{thm:eigenfunctions}, it follows by Proposition \ref{prop:removingef} that $f_{ij}^{(I)}$ is an eigenfunction of $G_I$ with eigenvalue
\begin{align*}
    1 + \frac{\lambda - 1}{1 - |I|\alpha} 
    &= 1 + \frac{k\alpha - 1}{1 - |I|\alpha} \\
    &= \frac{1 - |I|\alpha}{1 - |I|\alpha} + \frac{k\alpha - 1}{1 - |I|\alpha} \\
    &= (k - |I|)\cdot \frac{\alpha}{1 - \alpha |I|}.
\end{align*}
The statement follows by Theorem \ref{thm:eigenfunctions}.
\end{proof}

\section{Complement}\label{sec:comp}

The \emph{complement} of a graph $G$ is defined as the graph $\overline{G}$ with vertex set $V(\overline{G})\coloneqq V(G)$ and edge set
\[
\quad 
E(\overline G)\coloneqq \bigl\{\{u,v\}\colon u,v\in V(G)\ \text{distinct, and}\ \{u,v\}\notin E(G)\bigr\}.
\]

In this section, we fix a $d$-regular graph $G$ and investigate FAT colorings of its complement $\overline{G}$. After giving the relevant definitions, we establish a theorem that relates FAT colorings of $G$ to those of $\overline{G}$.

Before stating the main result of this section we recall that, by Theorem \ref{thm:2.12} (Theorem 2.12 in \cite{beersmulasFAT}), FAT colorings of connected $d$-regular graphs necessarily have coloring classes of equal size. We now extend the statement to the general (not necessarily connected) $d$-regular case.

\begin{theorem}\label{thm:FATreggen}
    Let $G$ be a $d$-regular graph. Any FAT $k$-coloring of $G$ with parameter $\alpha$ satisfies exactly one of the following:
    \begin{enumerate}
        \item $\alpha = 0$, in which case each connected component of $G$ is monochromatic, and every color appears in at least one component.
        \item $\alpha > 0$, in which case each connected component admits $k$ coloring classes of equal size, and therefore every coloring class of $G$ has size $N/k$.
    \end{enumerate}
\end{theorem}
\begin{proof}
    We distinguish the two cases appearing in the statement.
    \begin{enumerate}
        \item If $\alpha = 0$, then by definition no vertex has a neighbor of a different color. Consequently, each connected component must be monochromatic. Since the FAT coloring uses exactly $k$ colors, each color must occur in at least one connected component.
        \item If $\alpha > 0$, then by the definition of a FAT coloring, every vertex has neighbors in all coloring classes distinct from its own (in particular, this holds when $\alpha = 1/(k-1)$). Hence, every color is present in each connected component. Applying Theorem \ref{thm:2.12} to each component, we deduce that every component admits $k$ coloring classes of equal size. Summing over all components, it follows that each coloring class in $G$ has size $N/k$. \qedhere
    \end{enumerate}
\end{proof}
We apply Theorem \ref{thm:FATreggen} in the proof of the following theorem, which establishes a correspondence between FAT colorings of a $d$‑regular graph and FAT colorings of its complement.

\begin{theorem}\label{thm:comp}
    Let $G$ be a $d$-regular graph whose complement is connected, and let $V_1,\ldots,V_k$ be a partition of $V(G)$. Then, the following statements are equivalent:
    \begin{enumerate}[label=(\arabic*)]
        \item The partition $V_1,\ldots,V_k$ induces a FAT $k$-coloring of $G$ with parameter $\alpha$. Furthermore, when $\alpha = 0$, the corresponding coloring classes all have the same cardinality.
        \item The partition $V_1,\ldots,V_k$ induces a FAT $k$-coloring of $\overline G$ with parameter
        \[
        \frac{N/k-d\cdot \alpha}{N-d-1}.
        \]
    \end{enumerate}
\end{theorem}

\begin{proof}
    We first prove that $(1)$ implies $(2)$. Let $G$ be a $d$-regular graph, and let $V_1,\ldots,V_k$ be the coloring classes of a FAT $k$-coloring of $G$ with parameter $\alpha$. If $\alpha=0$, then by assumption all coloring classes have the same cardinality. If $\alpha>0$, then all coloring classes are equal-sized by Theorem \ref{thm:FATreggen}.
    
    Consider an arbitrary vertex $v \in V(\overline G) = V(G)$, and let $i \in \{1,\ldots,k\}$ be such that $v \in V_i$. For any index $j \in \{1,\ldots,k\}$ with $j \neq i$, the number of neighbors of $v$ lying in $V_j$ in the complement graph $\overline G$ is
    \[
    |V_j| - \alpha \deg_G v \;=\; \frac{N}{k} - \alpha d.
    \]
    Since $\overline G$ is an $(N-d-1)$-regular graph, it follows that the partition $V_1,\ldots,V_k$ yields a FAT $k$-coloring of $\overline G$ with parameter
    \[
        \frac{N/k - \alpha d}{N-d-1},
    \]
    which proves the first implication.
    
    For the converse implication, consider $\overline G$, which is regular and connected, and let $V_1,\ldots,V_k$ be the coloring classes of a FAT $k$-coloring of $\overline G$. By Theorem \ref{thm:2.12}, all coloring classes have the same cardinality, that is, $|V_i|=N/k$ for all $i=1,\ldots,k$.
    By an entirely analogous argument to the one above, now interchanging the roles of $G$ and $\overline G$, we conclude that the same partition $V_1,\ldots,V_k$ induces a FAT $k$-coloring of $G$ with the corresponding parameter.
\end{proof}

\begin{example}
    \begin{figure}
        \centering
        \includegraphics[scale=0.66]{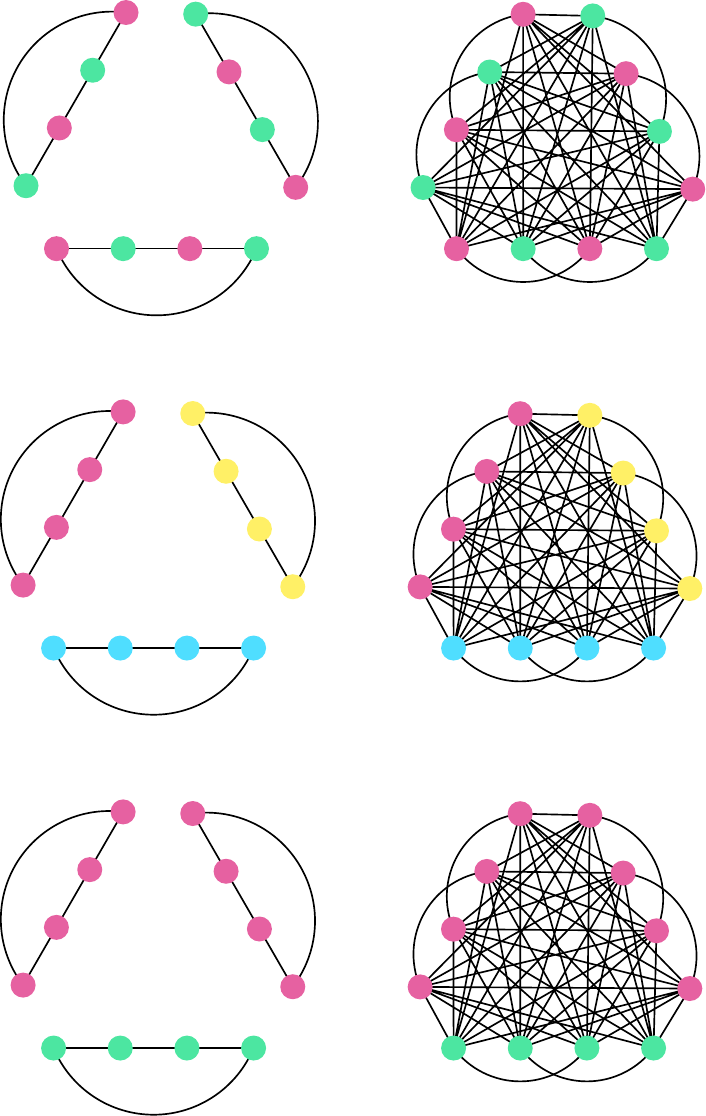}
        \caption{Three FAT colorings of $C_4\sqcup C_4\sqcup C_4$ and the corresponding induced coloring of its complement.}
        \label{fig:comp2}
    \end{figure}

    Let $C_4$ denote the $4$-cycle, and consider the disjoint union $C_4\sqcup C_4 \sqcup C_4$ of three such cycles. In Figure \ref{fig:comp2}, three pairwise distinct FAT colorings of $C_4\sqcup C_4 \sqcup C_4$ with parameter $\alpha = 0$ are illustrated.
    By Theorem \ref{thm:comp}, the top two FAT colorings induce FAT colorings of the complement, as their coloring classes are of equal cardinality. The corresponding FAT parameters are $\alpha = 0$ (top) and $\alpha = 4/9$ (middle). In contrast, the bottom FAT coloring does not induce a FAT coloring of the complement, since in this case the color classes are not equal-sized.
\end{example}

The previous example illustrates the necessity of the additional constraint $|V_i| = N/k$ in Theorem~\ref{thm:comp} when the graph $G$ is disconnected. In contrast, when $G$ is connected, the constraint is satisfied automatically by Theorem \ref{thm:2.12}, yielding the following corollary of Theorem \ref{thm:comp}.
\begin{corollary}
    Let $G$ be a connected $d$-regular graph with a connected complement $\overline{G}$. A partition $V_1,\ldots,V_k$ of the vertex set $V(G)$ induces a FAT $k$-coloring of $G$ with parameter $\alpha$ if and only if it induces a FAT $k$-coloring of $\overline G$ with parameter
    \[
        \frac{N/k - d\alpha}{N - d - 1}.
    \]
\end{corollary}

\section{Tensor, Cartesian and strong products}\label{sec:products}
We now investigate how FAT colorings behave under graph products. Throughout this section, we fix two graphs $G_1$ and $G_2$.

\begin{definition}
    The \emph{tensor product} (also called the \emph{Kronecker product}) of $G_1$ and $G_2$ is the graph $G_1\times G_2$ whose vertex and edge sets are given by
    \begin{align*}
        V(G_1\times G_2) &\coloneqq V(G_1) \times V(G_2),\\
        E(G_1\times G_2) &\coloneqq \biggl\{ \bigl\{(v_i,w_k),(v_j,w_l)\bigr\} \colon \{v_i,v_j\} \in E(G_1) \text{ and } \{w_k,w_l\} \in E(G_2)\biggr\}.
    \end{align*}
    The \emph{Cartesian product} of $G_1$ and $G_2$ is the graph $G_1\square G_2$ with
    \begin{align*}
        V(G_1 \square G_2) &= V(G_1) \times V(G_2),\\
        E(G_1 \square G_2) &= \bigl\{\{(u_1,u_2),(v_1,v_2)\} \colon u_1 = v_1 \text{ and } u_2\sim v_2 \text{ or } u_1 \sim v_1 \text{ and } u_2 = v_2\bigr\}.
    \end{align*}
    The \emph{strong product} of $G_1$ and $G_2$ is the graph $$G_1\boxtimes G_2 \coloneqq G_1\times G_2 \cup G_1\square G_2,$$ whose vertex set is $V(G_1)\times V(G_2)$ and whose edge set is the union of the edge sets of the tensor and Cartesian products.
\end{definition}

The remainder of this section is structured as follows. In Subsection \ref{subsec:spectralprop}, we recall several spectral properties of the tensor, Cartesian, and strong products, which are 
essentially due to Butler \cite{butler2016algebraic} and are stated here in a form that is adapted to our 
setting. In Subsection \ref{subsec:lifting}, we then use these properties to show that any FAT coloring of a graph $G_1$ can be extended to the graph products $G_1\times G_2$, $G_1\square G_2$, and $G_1\boxtimes G_2$, provided that $G_1$ and $G_2$ satisfy appropriate conditions.

\subsection{Spectral properties}\label{subsec:spectralprop}
In this subsection, we recall several spectral properties of graph products that were originally established by Butler \cite{butler2016algebraic}, and that will be applied in the next subsection. The first result describes how the eigenvalues and eigenfunctions of the normalized Laplacian behave under the tensor product of graphs.

\begin{proposition}[Theorem 2 from \cite{butler2016algebraic}]\label{prop:tensorspec}
    Let $G$ and $H$ be simple graphs without isolated vertices, and let $0=\lambda_1\leq \lambda_2\leq\cdots\leq\lambda_{N_1}$ and $0=\theta_1\leq\theta_2\leq\cdots\leq\theta_{N_2}$ denote the eigenvalues of the normalized Laplacian associated with $G$ and $H$, respectively. Then the eigenvalues of the normalized Laplacian of the tensor product $G_1\times G_2$ are given by
    \begin{equation*}
        \bigl\{\lambda_i+\theta_j-\lambda_i\theta_j \colon 1\leq i\leq N_1,\; 1\leq j\leq N_2\bigr\}.
    \end{equation*}
    Moreover, the eigenfunctions of $L(G_1\times G_2)$ are precisely the tensor products of the eigenfunctions of $L(G_1)$ and $L(G_2)$, respectively.
\end{proposition}

\begin{remark}\label{rmk:tensor}
Given a square matrix $M$, we let $\Spec(M)$ denote the multiset of its eigenvalues, counted with algebraic multiplicity. It is a classical result in matrix theory that, for any two square matrices $M_1$ and $M_2$,
    \[
    \Spec(M_1\otimes M_2) \;=\; \bigl\{\mu_1\mu_2 \colon \mu_i\in \Spec(M_i)\bigr\},
    \]
    and that the corresponding eigenfunctions of $M_1\otimes M_2$ associated with the eigenvalue $\mu_1\mu_2$ are precisely those of the form $f_1\otimes f_2$, where $f_i$ is an eigenfunction of $M_i$ with eigenvalue $\mu_i$ for $i=1,2$.
\end{remark}    

An analogous result can be established for the eigenfunctions corresponding to the Cartesian and strong products of two regular graphs. The following statement is adapted from Proposition 3 in \cite{butler2016algebraic}, where the eigenvalue formulas are established. We state it here in a form that also makes the associated eigenfunctions explicit, since we know from Theorem \ref{thm:eigenfunctions} that eigenfunctions are a crucial tool for the identification of FAT colorings. The proof follows the argument of Butler \cite{butler2016algebraic}, adapted to explicitly describe the corresponding eigenfunctions, and is analogous to the proof of the corresponding statement for the adjacency matrix given by Brouwer and Haemers \cite{BrouwerHaemers2012}. Both Proposition \ref{prop:tensorspec} and Proposition \ref{prop:Cartstrongeigen} are used in Subsection \ref{subsec:lifting} to prove two results concerning FAT colorings of the product graphs $G_1\times G_2$, $G_1\square G_2$, and $G_1\boxtimes G_2$.

\begin{proposition}\label{prop:Cartstrongeigen}
    Let $G_1$ and $G_2$ be $d^{(1)}$- and $d^{(2)}$-regular graphs, respectively. Then the eigenfunctions of the normalized Laplacian matrices $L(G_1\square G_2)$ and $L(G_1\boxtimes G_2)$ are precisely the tensor products $f_1\otimes f_2$, where $f_i$ is an eigenfunction of $L(G_i)$ for $i=1,2$. The corresponding eigenvalues of $L(G_1 \square G_2)$ are given by
\begin{equation*}
\biggl\{ \frac{d^{(1)}\lambda_i + d^{(2)} \theta_j}{d^{(1)}+d^{(2)}} \colon 1\leq i\leq N_1,\; 0 \leq j \leq N_2 \biggr\},
\end{equation*}
and the eigenvalues of $L(G_1\boxtimes G_2)$ are given by
\begin{equation*}
    \biggl\{ \frac{d^{(1)} d^{(2)}(\lambda_i+\theta_j - \lambda_i\theta_j) + d^{(1)}\lambda_i + d^{(2)}\theta_j}{d^{(1)}d^{(2)}+d^{(1)}+d^{(2)}}\colon 1\leq i\leq N_1,\; 0 \leq j \leq N_2 \biggr\}.
\end{equation*}
\end{proposition}
\begin{proof}
    By Propositions 1.4.6 and 1.4.8 in \cite{BrouwerHaemers2012}, the adjacency matrices of the Cartesian product $G_1\square G_2$ and the strong product $G_1\boxtimes G_2$ are given by
\begin{align*}
    A(G_1\square G_2) &= A(G_1)\otimes I + I\otimes A(G_2),\\
    A(G_1\boxtimes G_2) &= A(G_1)\otimes A(G_2) + A(G_1)\otimes I + I\otimes A(G_2).
\end{align*}
Since $G_1$ and $G_2$ are $d^{(1)}$- and $d^{(2)}$-regular, respectively, we see that
\begin{align*}
    D^{-1}A(G_1\square G_2) &= \frac{1}{d^{(1)}+d^{(2)}}\bigl(A(G_1)\otimes I + I\otimes A(G_2)\bigr),\\
    D^{-1}A(G_1\boxtimes G_2) &= \frac{1}{d^{(1)}+d^{(2)}+d^{(1)}d^{(2)}}\bigl(A(G_1)\otimes A(G_2) + A(G_1)\otimes I + I\otimes A(G_2)\bigr).
\end{align*}
It follows from the previous identities, together with Remark \ref{rmk:tensor}, that all four matrices $A(G_1\square G_2)$, $A(G_1\boxtimes G_2)$, $D^{-1}A(G_1\square G_2)$, and $D^{-1}A(G_1\boxtimes G_2)$ admit a common eigenbasis consisting of tensor products $f_1\otimes f_2$, where $f_i$ is an eigenfunction of $A(G_i)$, $D^{-1}A(G_i)$, and $L(G_i)$ for $i=1,2$. Consequently, the normalized Laplacian matrices $L(G_1\square G_2)$ and $L(G_1\times G_2)$ also have eigenfunctions of the form $f_1\otimes f_2$.

Now let $\lambda_1\leq \lambda_2\leq \cdots \leq \lambda_{N_1}$ and $\theta_1\leq \theta_2\leq \cdots \leq \theta_{N_2}$ denote the eigenvalues of $L(G_1)$ and $L(G_2)$, respectively.
By Proposition 3 in \cite{butler2016algebraic}, the eigenvalues of the normalized Laplacian of $G_1 \square G_2$ are
\begin{equation*}
\biggl\{ \frac{d^{(1)}\lambda_i + d^{(2)} \theta_j}{d^{(1)}+d^{(2)}} \colon 1\leq i\leq N_1,\, 1 \leq j \leq N_2 \biggr\},
\end{equation*}
and the eigenvalues of the normalized Laplacian of $G_1\boxtimes G_2$ are
\begin{equation*}
    \biggl\{ \frac{d^{(1)} d^{(2)}(\lambda_i+\theta_j - \lambda_i\theta_j) + d^{(1)}\lambda_i + d^{(2)}\theta_j}{d^{(1)}d^{(2)}+d^{(1)}+d^{(2)}}\colon 1\leq i\leq N_1,\, 1 \leq j \leq N_2 \biggr\}.\qedhere
\end{equation*}
\end{proof}

\subsection{Lifting FAT colorings}\label{subsec:lifting}
\begin{figure}
    \centering
    \includegraphics[scale=0.61]{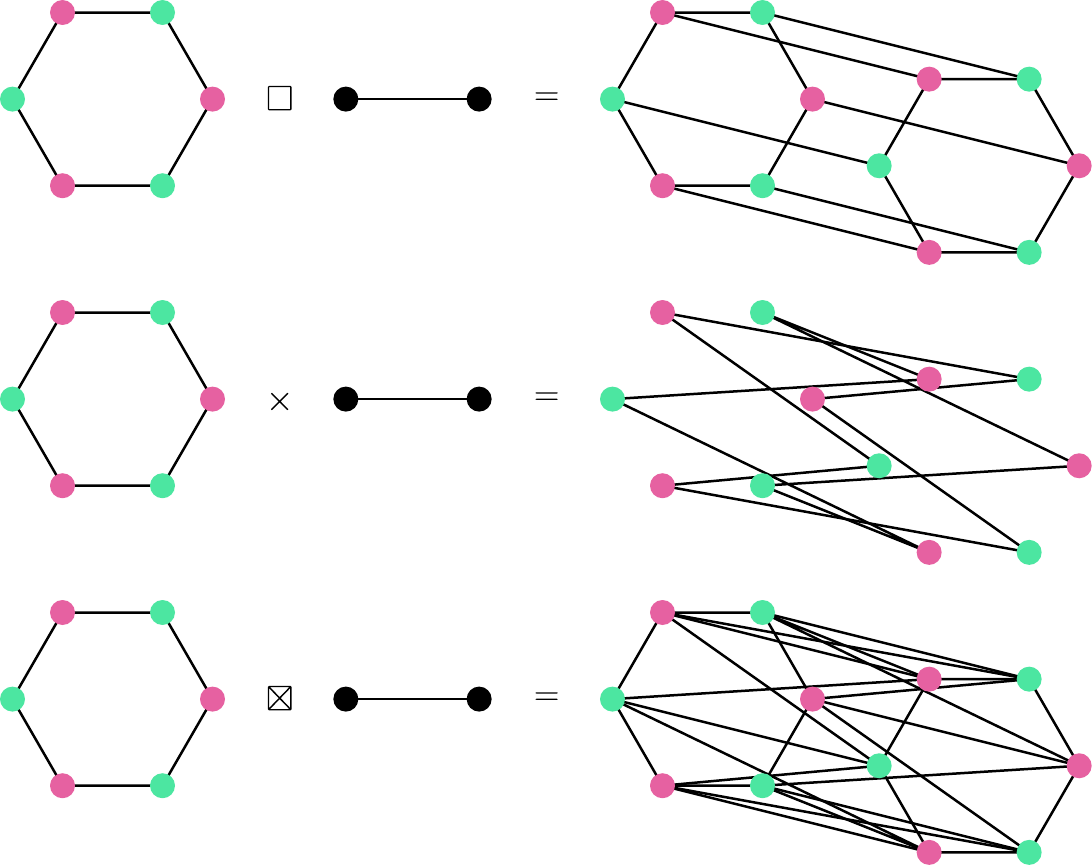}
    \caption{Lifting a $2$-coloring of $C_6$ to its Cartesian (top), tensor (middle), and strong (bottom) products with $K_2$.}
    \label{fig:products2}
\end{figure}

\begin{figure}
    \centering
    \includegraphics[scale=0.61]{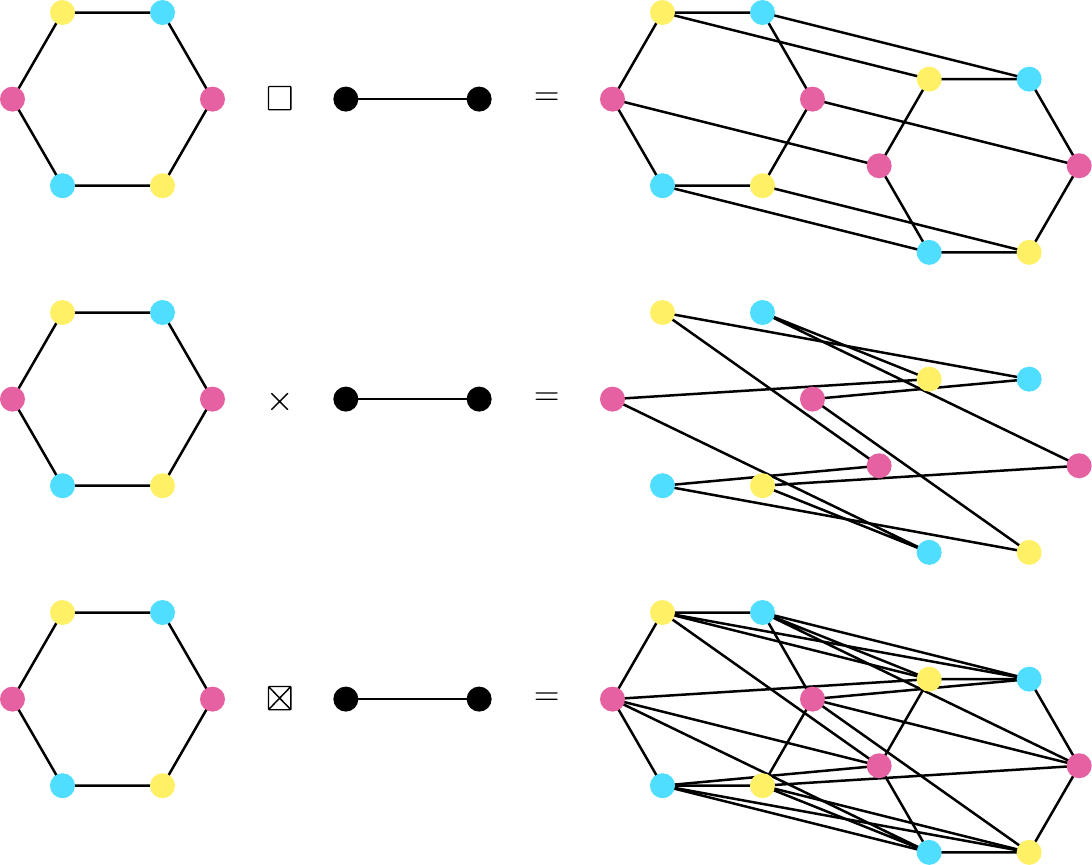}
    \caption{Lifting a $3$-coloring of $C_6$ to its Cartesian (top), tensor (middle), and strong (bottom) products with $K_2$.}
    \label{fig:products3}
\end{figure}

We now use Propositions \ref{prop:tensorspec} and \ref{prop:Cartstrongeigen} to prove two results concerning FAT colorings of the product graphs $G_1\times G_2$, $G_1\square G_2$, and $G_1\boxtimes G_2$. We formulate and prove the two theorems separately, due to the different conditions on the graphs $G_1$ and $G_2$, although their statements and proofs are analogous: one can lift a FAT coloring of $G_1$ to the product graphs $G_1\times G_2$, $G_1\square G_2$, and $G_1\boxtimes G_2$ (see Figures \ref{fig:products2} and \ref{fig:products3}).

\begin{theorem}\label{thm:tensor}
    Let $V_1,\ldots,V_k$ denote the coloring classes of $G_1$ with respect to a FAT $k$-coloring with parameter $\alpha$. Then, the tensor product $G_1\times G_2$ also admits a FAT $k$-coloring with parameter $\alpha$. Moreover, the corresponding coloring classes in $G_1\times G_2$ are the sets $V_i\times V(G_2)$.
\end{theorem}
\begin{proof}
    By Theorem \ref{thm:eigenfunctions}, the functions $f_{ij}^{(1)}$ are eigenfunctions of $G_1$ with eigenvalue $k\alpha$. Furthermore, by Proposition \ref{prop:tensorspec}, the functions $f_{ij}^{(1)}\otimes \boldsymbol {1}^{(2)}$ are eigenfunctions of the tensor product $G_1\times G_2$. Observe that
    \begin{equation*}
        f_{ij}^{(1)}\otimes \boldsymbol {1}^{(2)}(v,w) = \begin{cases}
            1 &\text{ if } (v, w) \in V_i\times V(G_2),\\
            -1 &\text{ if } (v,w)\in V_j\times V(G_2),\\
            0 &\text{ otherwise.}
        \end{cases}
    \end{equation*}
    Consequently, the function $f_{ij}^{(1)}\otimes \boldsymbol {1}^{(2)}$ coincides with the function $f_{ij}^{(1,2)}$ associated with the partition classes $V_i\times V(G_2)$ for $i=1,\ldots,k$. By Theorem \ref{thm:eigenfunctions}, the partition corresponds to a FAT $k$-coloring of $G_1\times G_2$.

    To determine the corresponding parameters, we apply Proposition \ref{prop:tensorspec}. The eigenvalue associated with the functions $f_{ij}^{(1,2)}$ is given by
    \[
    k\alpha + 0 - k\alpha \cdot 0 = k\alpha.
    \]
Therefore, by Theorem \ref{thm:eigenfunctions}, the partition classes $V_i\times V(G_2)$ induce a FAT $k$-coloring of $G_1\times G_2$ with parameter $\alpha$.
\end{proof}

\begin{theorem}\label{thm:Cartstrong}
    Let $G_1$ and $G_2$ be $d^{(1)}$- and $d^{(2)}$-regular graphs, respectively. 
    Let $V_1,\ldots,V_k$ denote the coloring classes of $G_1$ with respect to a FAT $k$-coloring with parameter $\alpha$. Then, $G_1\square G_2$ and $G_1 \boxtimes G_2$ each admit a FAT $k$-coloring with parameters
    \[
        \frac{\alpha d^{(1)}}{d^{(1)}+d^{(2)}} 
        \quad\text{and}\quad
        \frac{\alpha\bigl(d^{(1)} + d^{(1)}d^{(2)}\bigr)}{d^{(1)} + d^{(2)} + d^{(1)}d^{(2)}},
    \]
    respectively. In both cases, the coloring classes are given by $V_i\times V(G_2)$.
\end{theorem}
\begin{proof}
    By Theorem \ref{thm:eigenfunctions}, the functions $f_{ij}^{(1)}$ are eigenfunctions of $G_1$ with eigenvalue $k\alpha$. By Proposition \ref{prop:Cartstrongeigen}, the functions $f_{ij}^{(1)}\otimes \boldsymbol{1}^{(2)}$ are eigenfunctions of both $G_1\square G_2$ and $G_1 \boxtimes G_2$. Analogous to the proof of Theorem \ref{thm:tensor}, the function $f_{ij}^{(1)}\otimes \boldsymbol{1}^{(2)}$ coincides with the function $f_{ij}^{(1,2)}$ whose partition classes are given by $V_i\times V(G_2)$ for $i=1,\ldots,k$. By Theorem \ref{thm:eigenfunctions}, the partition gives rise to a FAT $k$-coloring of $G_1\square G_2$ and of $G_1 \boxtimes G_2$.

    To determine the corresponding parameters, we use Proposition \ref{prop:Cartstrongeigen}.
    \begin{itemize}
        \item For $G_1\square G_2$, the eigenvalue associated with the functions $f_{ij}^{(1,2)}$ is
        \[
        \frac{d^{(1)} \cdot k\alpha + d^{(2)}\cdot 0}{d^{(1)}+d^{(2)}} = k\cdot\frac{d^{(1)}\alpha}{d^{(1)}+d^{(2)}}.
        \]
        Consequently, the parameter corresponding to the FAT $k$-coloring of $G_1\square G_2$ with coloring classes $V_i\times V(G_2)$ equals $\alpha d^{(1)}/\bigl(d^{(1)}+d^{(2)}\bigr)$.
        \item For $G_1\boxtimes G_2$, the eigenvalue associated with the functions $f_{ij}^{(1,2)}$ is
        \[
        \frac{d^{(1)}d^{(2)}\bigl(k\alpha + 0 - k\alpha \cdot 0\bigr) + d^{(1)}\cdot k\alpha + d^{(2)}\cdot 0}{d^{(1)}d^{(2)}+d^{(1)}+d^{(2)}} 
        = k\cdot \frac{d^{(1)}d^{(2)}\cdot \alpha  + d^{(1)}\cdot \alpha }{d^{(1)}d^{(2)}+d^{(1)}+d^{(2)}}.
        \]
        Hence, the parameter corresponding to the FAT $k$-coloring of $G_1\boxtimes G_2$ with coloring classes $V_i\times V(G_2)$ is given by $\alpha\bigl(d^{(1)}+d^{(1)}d^{(2)}\bigr)/\bigl( d^{(1)}+d^{(2)}+d^{(1)}d^{(2)}\bigr)$.\qedhere
    \end{itemize}
\end{proof}

\section*{Funding}
Raffaella Mulas is supported by the Dutch Research Council (NWO) through the grant VI.Veni.232.002.

\bibliographystyle{plain} 
\bibliography{Bibliography}

\end{document}